\documentclass[preprint,12pt]{elsarticle}

\usepackage{amsfonts,amsmath,amssymb,amsthm}
\usepackage{color,graphicx,subfigure}
\usepackage[all]{xy}
\usepackage{enumerate}

\theoremstyle{plain}
\newtheorem{theorem}{Theorem}
\newtheorem{lemma}{Lemma}

\newtheorem{conj}{Conjecture}

\theoremstyle{definition}
\newtheorem{examp}{Example}

\newcommand{\divides}{\bigm|}
\newcommand{\ndivides}{%
  \mathrel{\mkern.5mu 
    \ooalign{\hidewidth$\big|$\hidewidth\cr$\nmid$\cr}%
  }%
}
\newtheorem{case}{Case}
\newcommand{\ndiv}{\hspace{-4pt}\not|\hspace{2pt}}

\journal{Discrete Applied Mathematics}

\begin{document}

\begin{frontmatter}

\title{Word Length Perturbations in Certain Symmetric Presentations of Dihedral Groups}

\author[label1]{Michael P.\ Allocca}
\author[label2]{Jason M.\ Graham}
\author[label3]{Candice R.\ Price}
\author[label4]{Shannon N.\ Talbott}
\author[label2]{Jennifer F.\ Vasquez}

\address[label1]{Department of Mathematics and Computer Science, Muhlenberg College, Allentown, PA, USA}
\address[label2]{Department of Mathematics, University of Scranton, Scranton,  PA, USA}
\address[label3]{Department of Mathematics, University of San Diego, San Diego, CA, USA}
\address[label4]{Department of Mathematics and Computer Science, Moravian College, Bethlehem, PA, USA}

\begin{abstract}
Given a finite group with a generating subset there is a well-established notion of length for a group element  given in terms of its minimal length expression as a product of elements from the generating set. Recently,  certain quantities called $\lambda_{1}$ and $\lambda_{2}$ have been defined that allow for a precise measure of how stable a group is under certain types of small perturbations in the generating expressions for the elements of the group. These quantities provide a means to measure differences among all possible paths in a Cayley graph for a group, establish a group theoretic analog for the notion of stability in nonlinear dynamical systems, and play an important role in the application of groups to computational genomics.  In this paper, we further expose the fundamental properties of  $\lambda_{1}$ and $\lambda_{2}$ by establishing their bounds when the underlying group is a dihedral group.   An essential step in our approach is to completely characterize so-called symmetric presentations of the dihedral groups, providing insight into the manner in which $\lambda_{1}$ and $\lambda_{2}$ interact with finite group presentations. This is of interest independent of the study of the quantities $\lambda_{1},\; \lambda_{2}$. Finally, we discuss several conjectures and open questions for future consideration.
\end{abstract}

\begin{keyword}
Dihedral groups, Group presentations, Word length, Generating set, Minimal length
\end{keyword}

\end{frontmatter}

\section{Introduction and Background}\label{sec:intro}

For a finite group, $G$, with generating set, $S$, there is a notion of length for any element $g\in G$: write $g$ as a product of elements (\emph{i.e.}, as a word) from $S$, using as few generators as possible. Then the length of $g$ is the number of generators appearing in a minimal expression (\emph{i.e.}, smallest word expression) for $g$. Furthermore, this notion of length provides a related notion of distance, or metric between two elements $g,g'$ of $G$. From a geometric perspective, this is related to distances along paths of the Cayley graph for $G$ associated with the generating set $S$.

Let $G$ be a finite group with symmetric generating set $S = \{s_{1},s_{2},\ldots,s_{n}\}$, where symmetric means $1 \not \in S$ and $s \in S \Rightarrow s^{-1} \in S$. One reason to consider symmetric generating sets is to ensure that a group element and its inverse have the same length. A word from $S$ in $G$ is any expression formed by taking products of elements in $S$. We refer to the elements of $S$ as letters. Then, for any element $g\in G$, we define the length $l_{S}(g)\in[0,\infty)$ of $g$ with respect to $S$ to be the minimal number of letters in $S$ for which $g$ can be written as a word from $S$. Note that $l_{S}(g) = 0 \Leftrightarrow g = 1$. Recently, the following quantities have been defined \cite{MoultonSteel2012}:
\begin{align}
  \lambda_{1}(G,S)&:=  \max_{g\in G,\ s\in S}\{l_{S}(gsg^{-1})\},\label{eq:lambdaOneDef} \\
  \lambda_{2}(G,S)&:=  \max_{g\in G,\ s,s'\in S}\{l_{S}(gss'g^{-1})\}.\label{eq:lambdaTwoDef}
\end{align}

The quantities $\lambda_{1},\lambda_{2}$ serve to precisely address the questions: given a word, what is the effect on its length after either the deletion of one letter (quantified by $\lambda_{1}$), or the replacement of one letter by another distinct letter (quantified by $\lambda_{2}$)? This presents an analogy between how large these measures can be, and the sensitivity of nonlinear dynamical systems to small perturbations in initial conditions, \emph{i.e.}, the so-called ``butterfly effect" \cite{MoultonSteel2012}.

There are several initial observations about $\lambda_{1}$ and $\lambda_{2}$ to note.  First, there is the bound
\begin{align}
  \lambda_{2}(G,S) \leq 2 \lambda_{1}(G,S), \label{eq:boundOne}
\end{align}
which holds for any group $G$ and symmetric generating set $S$, see \cite{MoultonSteel2012}. Second, the values for $\lambda_{1}$ and $\lambda_{2}$ in two extreme cases, either when $G$ is commutative, or when $S$ is as large as possible can easily be derived:

 Suppose that $G$ is a nontrivial commutative group and $S$ is any symmetric generating set. Then, $\lambda_{1}(G,S) = 1$ and $\lambda_{2}(G,S) \leq 2$. This follows because $sg = gs$ for any $g\in G$, $s\in S$ thus giving
  \begin{align*}
    l_{S}(gsg^{-1}) &= l_{S}(s) = 1.
  \end{align*}
Therefore, $\lambda_{1}(G,S) = 1$.

On the other hand,
  \begin{align*}
    l_{S}(gss'g^{-1}) &= l_{S}(ss') \leq 2.
  \end{align*}
Thus, we see that $\lambda_{2}(G,S) \leq 2$, and note that $\lambda_{2}(G,S)$ may be equal to 1 (for example, in the case that $ss'\in S$ for all $s,s'\in S$).

Now suppose that $G$ is a nontrivial finite group and $S = G - \{1\}$. That is, we take as our symmetric generating set all elements of $G$ except the identity. Then, for any $g\in G$, $s \in S$ we have that $g^{-1}sg = h$ for some $h \in G$. Now, either $h=1$ or $h\in S$. Thus, either $l_{S}(g^{-1}sg) = 0 \text{ or } 1$. Thus, we have that
  \[ \lambda_{1}(G,S) = \max_{g\in G, s\in S}l_{S}(g^{-1}sg)  = 1, \]
  whenever $S=G - \{1\}$.
The same basic argument can be used to show that $\lambda_{2}(G,S) = 1$ whenever $S=G - \{1\}$. Note that this example exhibits the feature that a large generating set $S$ gives small values of $\lambda_1$ and $\lambda_2$.

Finally, the main result of \cite{MoultonSteel2012}, Theorem 1 from that paper, establishes bounds on $\lambda_{1}$ and $\lambda_{2}$ in the cases where $G$ is the symmetric group $\Sigma_{n}$ of order $n!$ and $S$ is one of three distinct generating sets, the transpositions, the reversals, and the Coxeter generators. This together with the aforementioned observations already shows that the values of $\lambda_{1},\lambda_{2}$ are highly dependent on the specific choice of, and specific properties of both the group $G$ and the particular symmetric generating set $S$.

We note that an important source of motivation for the definitions of the $\lambda_{i}$, $i=1,2$, comes from computational approaches to the study of genome rearrangements. For some time now, combinatorial methods and finite groups such as permutation groups have played a major role in the modeling and exploration of problems arising in combinatorial genomics, for more on such topics see \cite{BafnaPevzner1996,Bergeron2005,EgriNagy2014,Hannenhalli1999,MoultonSteel2012}. The significance in relation to \cite{MoultonSteel2012} is that the notion of distance described there relates to the evolutionary distance between species based on differences in their respective genomes.

In addition, it is the case that many groups, including $\Sigma_{n}$, may be described as finitely presented groups \cite{coxeter1980,johnson1997,Smith2000}. Furthermore, the so-called group of circular permutations, which is of particular relevance to computational genomics \cite{EgriNagy2014}, can be described by adding a relation to the usual presentation for the affine symmetric group. Thus, when $G$ is a finitely presented group, there is interest in computing the values for $\lambda_{i}$, $i=1,2$ and investigating questions such as how do the relations in a presentation for $G$ affect the values of the $\lambda_{i}$, $i=1,2$. To date, very few efforts have been made in this direction.

In this work, we compute bounds or values for $\lambda_{1}$ and $\lambda_{2}$ for the dihedral groups $D_n$, $n>2$, which are of course well known to be noncommutative groups of order $2n$, see \emph{e.g.}\ \cite{DummitFoote2004,Smith2000}. Specifically, we consider dihedral groups given as finitely presented groups for several small generating sets that have the additional feature of being symmetric as described below, and also examine how the $\lambda_{i}$ values vary in both $n$ and the nature of the presentation. First, this involves developing a thorough understanding of the possible symmetric presentations for $D_{n}$, which we do completely for symmetric generating sets of cardinality less than or equal to three. Next, it must be established how the relations affect computing values or bounds for the $\lambda_{i}$. An important step along the way is to understand when two or more particular concrete realizations of a presentation in terms of specific elements of $D_{n}$ correspond to the same abstract presentation.

There are several reasons for considering the values of $\lambda_{1}$ and $\lambda_{2}$ for the dihedral groups. First, some computations are amenable to a direct approach which help to provide insight into certain aspects of the quantities $\lambda_{1}$ and $\lambda_{2}$. For example, when computing the values of $\lambda_{1}$ and $\lambda_{2}$ one is essentially concerned with the lengths of conjugates of elements of $S$ and conjugates of products of pairs of elements of $S$. That is, we want to know the lengths of elements of the form $g^{-1}sg$ and $g^{-1}ss'g$ where $g\in G,\ s,s'\in S$. Now, it is not necessarily the case that this produces every element of the group. This can already been seen in the previous example where $G$ is commutative and $S$ is any symmetric generating set. The computations contained herein for the dihedral groups further illustrate this point. Second, the geometric properties of the dihedral groups \cite{carter2009} provide intuition valuable in the computation of the quantities $\lambda_{1}$ and $\lambda_{2}$.

The remainder of the paper is organized as follows: In section \ref{presClass}, we develop a complete classification of all symmetric presentations of $D_{n}$ with the cardinality of the generating set less than or equal to three. In section \ref{vonDyckIso}, we establish some intermediate results used in the construction of bounds for the values of the $\lambda_i$ that are also of interest independent of the quantities $\lambda_i$. In section \ref{lambdaBounds}, we carry out the construction of the bounds for $\lambda_i$ using the aforementioned presentations, thus establishing the main results of the paper.   After concluding remarks in section \ref{conclusion}, we provide ancillary results connected to section \ref{presClass} in an appendix.

\section{Classification of Small Presentations for $D_{n}$}
\label{presClass}

In this section, we give a complete classification of all symmetric presentations of the dihedral groups with the cardinality of the generating set less than or equal to three. This is carried out largely by looking at concrete realizations of order two and order three symmetric generating subsets of $D_{n}$, with its standard description, see (\ref{eq:DnDescription}).

We begin by recalling some basic features of the dihedral groups in order to establish notation and perspective. The dihedral group $D_{n}$ is the group of order $2n$ with the following properties: $D_{n}$ contains an order $n$ cyclic subgroup and $n$ distinct elements, each of order two, none of which belongs to the order $n$ cyclic subgroup. From a geometric perspective $D_{n}$ represents the symmetries of a regular $n$-gon under rigid motions. Often, $D_{n}$ is described by the presentation
\begin{equation}
D_{n} = \left<r,f \mid r^{n}=f^{2}=1 , rf=fr^{-1}  \right>, \label{eq:DnPresA}
\end{equation}
\emph{e.g.}\ \cite{DummitFoote2004}. Then, the dihedral groups can be written out as the $2n$ distinct elements
\begin{align}
   D_{n}&=\{1,r,r^2,\ldots , r^{n-1},f,rf,r^2f,\ldots,r^{n-1}f\}. \label{eq:DnDescription}
\end{align}
We also have that
\begin{align}
  (r^{i})^{-1}&=r^{n-i}, \label{eq:rInverses} \\
  rf&=fr^{-1}=fr^{n-1}, \label{eq:noncommute}\\
  (r^{i}f)(r^{i}f) &= 1,\ \text{for any integer value of $0\leq i \leq n$}. \label{eq:flipSymmetry}
\end{align}
 We refer to the elements of the form $r^{i}f$ as involutions, or involution elements. In fact, these are the only involution elements besides $r^{\frac{n}{2}}$, which only exists if $n$ is even. See \cite{DummitFoote2004} for more information regarding the dihedral groups, which we note can also be viewed as a special class of Coxeter groups \cite{Humphreys1990}. All of the equations (\ref{eq:rInverses})-(\ref{eq:flipSymmetry}) will be used to facilitate the computations in the next sections.

Note that, while $\{r,f\}$ generates $D_{n}$, it is not symmetric. Thus, we now proceed to classify all presentations of $D_{n}$ where the generating set $S$ has cardinality less or equal to three and satisfies the symmetry condition that $s\in S \Rightarrow s^{-1} \in S$. We begin by classifying all symmetric generating sets with cardinality two, since there is no singleton $\{x\}$ such that there are relations on $x$ leading to a presentation of $D_{n}$.

\subsection{Classification of symmetric presentations using two elements}
\label{classSectionA}

Consider a set $S = \{x,y\}$. If $S$ is to be symmetric, then we must have either $xy=yx=1$, or $x^2=y^2=1$. However, in the first instance, $S = \{x,x^{-1}\}$ which only generates a cyclic group $\langle x \rangle$ and therefore is not $D_{n}$. Thus, if $S = \{x,y\}$ is a symmetric generating set for $D_{n}$ we must have that $x^2=y^2=1$. This implies that the product $xy$ must generate the order $n$ cyclic subgroup.

There is still the possibility that $x$ or $y$ is an element of order two and in the order $n$ cyclic subgroup of $D_{n}$. This occurs if and only if one of them is $r^{\frac{n}{2}}$ in the previous notation, which may only occur in the case that $n$ is even. But, this would then imply that $xy=yx$ which again does not generate $D_{n}$. Therefore, in the notation of (\ref{eq:DnDescription}), there are integers $a,b$ such that $x=r^{a}f$ and $y=r^{b}f$. Considering this leads to the following:

\begin{theorem}
\label{DnPres1b}
  $D_{n} = \left< x,y\mid \ x^{2} = y^{2} = 1, \ |xy| = n \right>$ completely classifies the symmetric presentations of $D_{n}$ with generating set $S=\{x,y\}$ of cardinality two.
\end{theorem}

\begin{proof}
  The elements $(xy),(xy)^{2},(xy)^{3},\cdots , (xy)^{n}=1$ produce the order $n$ cyclic subgroup of $D_{n}$.
  Then $(xy)x,(xy)^{2}x,(xy)^{3}x,\cdots , (xy)^{n}x$ are all distinct, as are $(xy)y,(xy)^{2}y,(xy)^{3}y,\cdots , (xy)^{n}y$.
  Now, we claim that there is an integer $p$ so that $y = (xy)^{p}x$. We have that $(xy)^{n} = (xy)^{n-1}xy = 1$. But then multiplying both of sides of the second equality by $y$ on the right gives $(xy)^{n-1}x = y$. Thus, the sets $\{(xy)x,(xy)^{2}x,(xy)^{3}x,\cdots , (xy)^{n}x\}$ and $\{(xy)y,(xy)^{2}y,(xy)^{3}y,\cdots , (xy)^{n}y\}$ are really the same. This yields another $n$ elements, each of which is its own inverse by the relations $x^2=y^2=1$. This is illustrated by the expansion
  \[ (xy)^p x (xy)^p x = (xy)(xy)\cdots(xy)x(xy)(xy)\cdots(xy)x , \]
  which, by the relations, reduces to $1$. Further, since $x^2=y^2=1$ and $xy\neq1,yx\neq1$, there are no other elements because words must alternate. Thus, this is a group with $2n$ elements, exactly half of which form a cyclic subgroup and the remaining are all involutions.
\end{proof}

As we will see, the result from Theorem \ref{DnPres1b} will be important for much of what follows. An interesting concrete case to consider is that of $D_4$, the set $S=\{f,r^{2}f\}$ does not generate the group. In fact, it is easy to see that the set $\{f,r^{a}f\}$ generates $D_{n}$ if and only if $a$ is relatively prime with $n$.

\subsection{Classification of symmetric presentations using three elements}
 \label{classSectionB}
  In this section, we categorize the symmetric presentations for $D_{n}$ that use three elements as the generating set. These are distinguished by the types of relations that are necessary to impose and are motivated by the examples that accompany the results.

\subsubsection{Two involutions and one cyclic element}
\label{classSectionB1}
\begin{examp}
  Let $n=6$, and consider the generating set $S=\{f,rf,r^{3}\}$. In this example, the pair $f,rf$ generates an order $n$ cyclic subgroup, and in fact all of $D_{6}$. The additional involution element $r^{3}$ is not necessary to obtain $D_{6}$ but, as we will see later, it may affect the values of the $\lambda_{i}$.
\end{examp}

\begin{theorem}
\label{DnPres2}
 For $n$ even,
 \[D_{n} = \left<  x,y,z\mid \ x^2=y^2=z^2=1,\ |xy|=n,\ z=(xy)^{\frac{n}{2}},\ xz=zx,\ yz=zy \right>\]
 classifies the symmetric presentations of $D_{n}$  using a generating set with three elements in the case where exactly one of them is an element of the order $n$ cyclic subgroup.
\end{theorem}

Note that there is no analogous presentation for $D_{n}$ when $n$ is odd since in this case the center of the group is trivial. The proof that this gives a presentation of $D_{n}$ is almost identical to that of Theorem \ref{DnPres1b}. In fact, the element corresponding to $r^{\frac{n}{2}}$ is not needed in order to generate the group but may affect the values of the $\lambda_{i}$ as we will see in section \ref{lambdaBoundsB1}.

\subsubsection{Two cyclic elements and one involution}
\label{classSectionB2}
\begin{examp}
  Let $n=6$, and consider the generating set $S=\{f,r,r^{5}\}$. In this example, the pair $f,r$ gives the standard presentation of $D_{6}$. The element $r^{5}$, being the inverse of $r$, is added simply to make the set symmetric.
\end{examp}

\begin{theorem}
\label{DnPres3}
  For a positive integer $n$,
  \[D_{n} = \left<x,y,z\mid \ x^2=1,\ |y|=|z|=n,\ yx=xz,\ yz=zy=1,\ |xy|=|xz|=2 \right>\]
  classifies the symmetric presentations of $D_{n}$  using a generating set with three elements in the case where exactly two of the elements are contained in the order $n$ cyclic subgroup.
\end{theorem}

Notice that the second and fourth relations imply that $z=y^{n-1}$. Thus, this is exactly the usual presentation for $D_{n}$ obtained by adding the inverse of $r$ thereby making $S$ symmetric. Then, it is clear that this gives a presentation for $D_{n}$.

\subsubsection{Three involution elements}
\label{classSectionB3}
Before describing in detail the remainder of the three element presentations for $D_{n}$ we present a simple result that will show that we do in fact end up with a complete classification. The basic idea is to understand what all of the possibilities are in the case that the three generators are all involutions, none of which belong to the order $n$ cyclic subgroup.

Let $a$ and $b$ be nonnegative integers less than $n$, then for a subset $S = \{f,r^{a}f,r^{b}f\}$ of $D_{n}$, define the cyclic subgroups $H_{1} = \langle r^{a}f f\rangle = \langle r^{a}\rangle$, $H_{2} = \langle r^{b}f f\rangle = \langle r^{b}\rangle$, and $H_{3} = \langle r^{b}fr^{a}f \rangle = \langle r^{b-a} \rangle$. Also, define the sets $H_{1}H_{2}H_{3} = \{h_{1}h_{2}h_{3}\mid \ h_{i}\in H_{i}, i=1,2,3\}$ and $H_{1}H_{2} = \{h_{1}h_{2}\mid \ h_{i}\in H_{i}, i=1,2\}$.

  We make the following observations:
  \begin{enumerate}
    \item For each $i=1,2,3$, $H_{i}$ is a normal subgroup of $D_{n}$ since it is generated by an element of the order $n$ cyclic subgroup of $D_n$. For example, take $H_{1}$, then
       \[ r^{k}r^{ja}r^{-k} = r^{ja} \in H_{1}, \]
       \[ r^{k}fr^{ja}r^{k}f = r^{-ja} \in H_{1}. \]
    \item $H_{3} \subset H_{1}H_{2}$. This follows since anything of the form $r^{j(b-a)}$, which is a typical element of $H_{3}$, is a special case of something of the form $r^{nb+ma}$, which is a typical element of $H_{1}H_{2}$. Thus, we have that $H_{1}H_{2}H_{3}=H_{1}H_{2}$.
  \end{enumerate}

  With this notation, we have the following lemma which provides a concise characterization of when three involutions will generate $D_{n}$.

  \begin{lemma}
  \label{subgroupLemma}
    For $a$ and $b$ be nonnegative integers less than $n$, let $S = \{f,r^{a}f,r^{b}f\} \subset D_{n} = \{1,r,\ldots,r^{n-1},f,rf,\ldots,r^{n-1}f\}$. Then, $S$ generates $D_{n}$ if and only if $H_{1}H_{2} = \{h_{1}h_{2}\mid \ h_{i}\in H_{i}\} = \langle r\rangle = \{1,r,\ldots,r^{n-1}\}$. By the previous observation, this is equivalent to the condition that $H_{1}H_{2}H_{3} = \langle r \rangle$.
  \end{lemma}

    \begin{proof}First suppose that $S = \{f,r^{a}f,r^{b}f\}$ generates $D_{n}$. Then there are integers $j,k,l$ such that\footnote{The reason that this is true is because $r$ must be made up of an alternating product of the elements from $S$ of even length. Then one can apply the pairwise commutativity property $abcd=cdab$ to rewrite this in the form $(r^{a}f f)^{j}(r^{b} f f)^{k}(r^{b}f)^{l}$.}
  \[ r = (r^{a}f f)^{j}(r^{b} f f)^{k}(r^{b}f)^{l} \in H_{1}H_{2}H_{3}=H_{1}H_{2}. \]
  This implies that $\langle r \rangle \subset H_{1}H_{2}$. On the other hand, by definition $H_{1}H_{2} \subset \langle r \rangle$. Thus, we have that $H_{1}H_{2} = \langle r \rangle$ and this proves the first direction.

  Conversely, suppose that $H_{1}H_{2} = \langle r \rangle$, then we have the order $n$ cyclic subgroup $\{1,r,\ldots, r^{n-1}\}$ of $D_{n}$. Note that the set $H_{1}H_{2}f = \{h_{1}h_{2}f\mid \ h_{i}\in H_{i}\}$ then gives exactly $n$ distinct involution elements.
  \end{proof}

   In order to obtain a slight generalization, suppose that we have a set $S = \{r^{a}f,r^{b}f,r^{c}f\}$, then take $S' = \{r^{a}f,r^{b-a}(r^{a}f),r^{c-a}(r^{a}f)\} = \{f',r^{a'}f',r^{b'}f'\}$, where $f'=r^{a}f$, $a'=b-a$, and $b'=c-a$. Then one can apply the previous result to $S'$, since it is equivalent to $S$.

   A consequence of Lemma \ref{subgroupLemma} is that, given a subset $S = \{f,r^{a}f,r^{b}f\}$ of $D_{n}$, there are five possibilities with respect to the generating of $D_{n}$. Using the notation previously introduced, these are:
   \begin{enumerate}
     \item $H_{1}=H_{2}=H_{3}=\langle r \rangle$
     \item Exactly two of $H_{1},H_{2},H_{3}$ are $\langle r \rangle $
     \item Exactly one of $H_{1},H_{2},H_{3}$ is $\langle r \rangle$
     \item None of $H_{1},H_{2},H_{3}$ is $\langle r \rangle$ but $H_{1}H_{2} = \langle r \rangle$
     \item $H_{1}H_{2}$ is a strict subgroup of $\langle r \rangle$, in other words $S = \{f,r^{a}f,r^{b}f\}$ does not generate $D_n$.
   \end{enumerate}
   This motivates the remainder of our classification of three element presentations of $D_{n}$. It is appropriate to note here that there are precise number-theoretic conditions for the exponents $a,b$ in $S = \{f,r^{a}f,r^{b}f\}$ that can be used to distinguish the five cases above in concrete situations. The details of this are given in Section \ref{NumThrConditions} so as not to detract from the main line of reasoning.

\begin{theorem}
\label{threeFlipsMain}
The following complete the classification of the presentations of $D_{n}$ with generating set of cardinality three.
\begin{enumerate}[(A)]
  \item \label{DnPres4}
  For a positive integer $n$, $D_{n}$  is $\left<x,y,z\right>$ with relations $x^2=y^2=z^2=1,\ |xy|= |xz| = |yz|= n$; $\forall a,b,c,d\in\{x,y,z\}$, $abcd=cdab$;  and $\langle xy \rangle = \langle xz \rangle  = \langle yz \rangle $.
  \item \label{DnPres5}
  For a positive integer $n$, $D_{n}$  is $\left<x,y,z\right>$ with relations $x^2=y^2=z^2=1,\ |xy|= |xz|=n$;  $\forall a,b,c,d\in\{x,y,z\}$, $abcd=cdab$;  and $\langle yz \rangle \subsetneq \langle xy \rangle = \langle xz \rangle$.
  \item  \label{DnPres6}
 For a positive integer $n$, $D_{n}$  is $\left<x,y,z\right>$ with relations $x^2=y^2=z^2=1,\ |xy|=n$;  $\forall a,b,c,d\in\{x,y,z\}$, $abcd=cdab$; and $\langle xz \rangle \subsetneq \langle xy \rangle$, $\langle yz \rangle \subsetneq \langle xy \rangle$.
 \item  \label{DnPres7}
 For a positive integer $n$, $D_{n}$  is $\left<x,y,z\right>$ with relations $x^2=y^2=z^2=1$; $\forall a,b,c,d\in\{x,y,z\}$, $abcd=cdab$; and if $H_{1} = \langle xy \rangle$, $H_{2} = \langle xz \rangle$, $H_{3} = \langle yz \rangle$ we have $H_{1}H_{2}$ is an order $n$ cyclic subgroup with $H_{1} \subsetneq H_{1}H_{2}$, $H_{2} \subsetneq H_{1}H_{2}$, and $H_{3} \subsetneq H_{1}H_{2}$.
\end{enumerate}
\end{theorem}

\begin{proof}
  In each of the cases (\ref{DnPres4})-(\ref{DnPres7}), it is clear that we obtain, at least $n$ elements of a cyclic subgroup, and at least $n$ distinct involution elements. Furthermore, each of the four presentations gives elements that are either contained in an order $n$ cyclic subgroup, or are involutions. It remains to verify that there are exactly $n$ elements that form a cyclic subgroup, and exactly $n$ distinct involution elements. This is precisely what is given by the final condition in the statement of each of (\ref{DnPres4})-(\ref{DnPres7}).
\end{proof}

Next, we provide examples to illustrate each of the cases described in Theorem \ref{threeFlipsMain} (\ref{DnPres4}) - (\ref{DnPres7}).

\begin{examp}
\label{threeFlipsExamples}
  The following examples illustrate the presentations given in Theorem \ref{threeFlipsMain} (\ref{DnPres4})--(\ref{DnPres7}).
  \begin{enumerate}[(A)]
    \item Let $n=5$ and take $S = \{f,rf,r^{2}f\}$, here all three pairs $\{f,rf\},\{f,r^{2}f\},\{rf,r^{2}f\}$ generate $D_{5}$.
    \item Let $n=6$ and take $S = \{f,rf,r^{2}f\}$, here only the two pairs $\{f,rf\},\{rf,r^{2}f\}$ generate $D_{6}$.
    \item  Let $n=6$ and take $S = \{f,rf,f^{4}f\}$, here only the pair $\{r,rf\}$ generates $D_{6}$.
    \item  Let $n=30$ and take $S=\{f,r^{3}f,r^{5}f\}$, here no pair generates $D_{30}$, but the triple $\{f,r^{3}f,r^{5}f\}$ does.
  \end{enumerate}
\end{examp}

Our problem now is to compute bounds on the values of $\lambda_{1}(D_{n},S)$ and $\lambda_{2}(D_{n},S)$ whenever $S$ is one of the generating sets for any of the presentations of $D_{n}$ that have been described in this section. First, we establish a technical lemma, Lemma \ref{lengthLemma}, in the next section that will allow us to considerably reduce the workload.

\section{Automorphisms of $D_{n}$ that preserve the $\lambda_{i}$ values}
\label{vonDyckIso}

In this section we construct certain automorphisms on the dihedral groups, represented as finitely presented groups, that have the feature of preserving the values of $\lambda_{1}(G,S)$ and $\lambda_{2}(G,S)$. These will serve as an important tool in our computation of bounds for $\lambda_{1}(G,S)$ and $\lambda_{2}(G,S)$ in section \ref{lambdaBounds}. For a fuller discussion regarding $D_{n}$ automorphisms, see \cite{cunningham2014}. To motivate the results of this section, consider the following example for $n=3$:
\begin{examp}
  Let $S_{1}=\{f,rf\}$ and $S_{2}=\{f,r^2f\}$ be two different generating sets for $D_{3}=\{1,r,r^2,f,rf,r^2f\}$. We define a map $S_{1}\rightarrow S_{2}$ by $f\mapsto f$, $rf\mapsto r^{2}f$. This produces an automorphism on $D_{3}$ by extending the mapping to the set of words from $S_{i}$, which sends generators to generators.  Observe:
  \begin{align}
    1 \ \ &\mapsto \ \ 1, \label{eq:relPres}\\
    r=(rf)(f) \ \ &\mapsto \ \ (r^{2}f)(f)=r^{2}, \\
    r^2=(f)(rf) \ \ &\mapsto \ \ (f)(r^{2}f)=r, \\
    f \ \ &\mapsto \ \ f, \\
    rf \ \ &\mapsto \ \ r^{2}f, \\
    r^{2}f=(f)(rf)(f) \ \ &\mapsto \ \ (f)(r^2f)(f)=rf.
  \end{align}
\end{examp}

Notice that in this example, the lengths of words are preserved under the automorphism on $D_{3}$, which is essentially defined by sending $f\mapsto f$, $rf\mapsto r^{2}f$, and then extending to words\footnote{The mapping $f\mapsto r^{2}f$ and $rf\mapsto f$ would also work. This point is addressed further at the end of this section.}. Another important point, discussed in greater detail below, is that the mapping just described preserves all of the important relations between elements in $S_{1}$.  In this section, we prove that this phenomenon generalizes in a specific manner.

We begin with a lemma that we refer to as the Van Dyck isomorphism property. In the following we denote by $\iota$ the usual inclusion map on a set.
\begin{lemma}
 \label{isoLemma}
  For two symmetric generating sets $S_{1} = \{x,y\}$ and $S_{2} = \{x',y'\}$ of $D_{n}$ that share all of the same relations, the mapping $\alpha:S_{1}\rightarrow S_{2}$ defined by $\alpha(x) = x'$, and $\alpha(y) = y'$ extends to a unique automorphism $\Phi:D_{n} \rightarrow D_{n}$ that maps generators to generators.
\end{lemma}

 \begin{proof}Let $F(S)$ denote the free group on the set $S$ and let $R$ denote the normal subgroup of $F(S)$ determined by any relations set on the elements of $S$. Then, the universal property for free groups leads to a unique homomorphism $\phi:F(S_{1}) \rightarrow F(S_{2})$ , which makes the following diagram commute.

\begin{displaymath}
    \xymatrix{ S_1 \ar@{^{(}->}[d]_{\iota} \ar[r]^{\alpha} & S_2 \ar@{^{(}->}[r]^{\iota} &F(S_2)& \\
               F(S_1) \ar@{.>}[urr]_\phi  & }
\end{displaymath}
 This homomorphism is defined by the formula
  \begin{align}
    &\phi\left(\left(x\right)^{\epsilon_{1}}\left(y\right)^{\delta_{1}}\cdots \left(x\right)^{\epsilon_{k}}\left(y\right)^{\delta_{k}} \right)  \\
    &= \left(x'\right)^{\epsilon_{1}}\left(y'\right)^{\delta_{1}}\cdots \left(x'\right)^{\epsilon_{k}}\left(y'\right)^{\delta_{k}},
  \end{align}
 where $\epsilon_{i},\delta_{i}\in \mathbb{Z}$. Next, consider the map $\pi \circ \phi :F(S_{1})\rightarrow F(S_{2})/R_{2}$, where $\pi$ is the canonical projection. The kernel of this map is $R_{1}$ since $S_{1}$ and $S_{2}$ have the same relations. Then the first isomorphism theorem applied to  $\pi \circ \phi$ gives a unique isomorphism $\Phi:F(S_{1})/R_{1}\rightarrow F(S_{2})/R_{2}$ such that the following diagram commutes.

\begin{displaymath}
    \xymatrix{ F(S_1) \ar[d]^\pi \ar[r]^\phi & F(S_2) \ar[r]^\pi &F(S_2)/R_2& \\
               F(S_1)/R_1 \ar@{.>}[urr]_\Phi  & }
\end{displaymath}
 By defining a map $\beta:S_{2}\rightarrow S_{1}$ by: $\beta(x') = x$, and $\beta(y') = y$ and interchanging $S_{1}$ and $S_{2}$, we obtain maps $\psi$ and $\Psi$ that are the inverses of $\phi$ and $\Phi$ respectively. Altogether we obtain the following commutative diagram
\begin{displaymath}
    \xymatrix{ S_1 \ar@{^{(}->}[d]_{\iota} \ar@<1.ex>[r]^{\alpha}
        & S_2  \ar@<1.ex>[l]^{\beta} \ar@{^{(}->}[d]^{\iota}  \\
               F(S_1)\ar[d]^\pi \ar@<1.ex>[r]^\phi
       & F(S_2)\ar[d]^\pi  \ar@<1.ex>[l]^\psi \\
               F(S_1)/R_1 \ar@<1.ex>[r]^\Phi
       & F(S_2)/R_2  \ar@<1.ex>[l]^\Psi  }
\end{displaymath}
from which the conclusion of the lemma follows.
\end{proof}

Using Lemma \ref{isoLemma}, we now establish when the homomorphism is also length-preserving.  This will serve as an important tool in establishing bounds on the quantities $\lambda_1(G,S)$ and $\lambda_2(G,S)$.

\begin{lemma}
 \label{lengthLemma}
If $S_{1} = \{x,y\}$ and $S_{2} = \{x',y'\}$ are symmetric generating sets of $D_{n}$ that share all of the same relations, then
  \[\max_{g\in D_{n}}\left\{l_{S_{1}}(g)\right\} =  \max_{g\in D_{n}}\left\{l_{S_{2}}(g)\right\}.\]
\end{lemma}

 \begin{proof}Let $\alpha: S_{1}\rightarrow S_{2}$ and $\Psi:D_{n} \rightarrow D_{n}$ be as in the previous lemma. Let $h\in D_{n}$ such that $l_{S_{1}}(h)=M_{1}:=\max_{g\in D_{n}}\left\{l_{S_{1}}(g)\right\}$, and let $\bar{h}\in D_{n}$ such that $l_{S_{2}}(\bar{h})=M_{2}:=\max_{g\in D_{n}}\left\{l_{S_{2}}(g)\right\}$.  Now suppose that $M_{2} > M_{1}$. We claim that this implies that there is no element $g\in D_{n}$ such that $\Psi(g)=\bar{h}$ thereby implying that $\Psi$ is not an automorphism in contradiction to the previous lemma. If there were such a $g$, then we can write it as a product of generators with no more that $M_{1}$ terms. But then since $\Psi$ is a homomorphism that maps generators to generators this would be an element with length in $S_{2}$ less than or equal to $M_{1}$ contradicting that $M_{2} > M_{1}$. By switching the roles of $S_{1}$ and $M_{1}$ with that of  $S_{2}$ and $M_{2}$ we obtain the result. \end{proof}

{\bf Remark:} Lemma \ref{isoLemma} is related to a result in the theory of group presentations sometimes known as Van Dyck's theorem, see \cite{fenrick1998, johnson1997,Smith2000}. Furthermore, there is nothing particularly special about the role of $D_{n}$, or the sizes of $S_{1}$ and $S_{2}$, in Lemma \ref{isoLemma} since the universal property for free groups and the first isomorphism theorem apply in more general settings. Thus, Lemma \ref{isoLemma} generalizes in an obvious way to give rise to a more general version of Lemma \ref{lengthLemma}. However, as we exhibit with an example, one must take care to be sure that a mapping from one specific generating set to another does in fact preserve all relations.

\begin{examp}First, consider the situation with $n=7$, $S_{1} = \{x_{1}=f,y_{1}=rf,z_{1}=r^2f\}$ and $S_{2} = \{x_{2}=f,y_{2}=r^3f,z_{2}=r^4f\}$. The mapping $f \mapsto f$, $rf \mapsto r^3f$, $r^2f \mapsto r^{4}f$ does not preserve the relation $z_{1}=(y_{1}x_{1})^{2}x_{1}=(y_{1}x_{1})y_{1}$ since in $S_{1}$ we have $(rf f)^{2}f = r^2f$ but in $S_{2}$ we have $z_{2} = r^4f = (y_{2}x_{2})^{4}x_{2} = (r^3f f)^{4}f$. On the other hand, it can be shown that the mapping $S_{1}=\{x_{1}=f,y_{1}=rf,z_{1}=r^2f\}\rightarrow S_{2}=\{x_{2}=r^3 f,y_{2}=f,z_{2}=r^4f\}$ given by $f \mapsto r^3 f$, $rf\mapsto f$, $r^2 f \mapsto r^4 f$ does preserve all relations, and hence preserves the $\lambda_{i}$ values. In contrast, for $n=9$, $S_{1} = \{x_{1}=f,y_{1}=rf,z_{1}=r^2 f\}$ and $S_{2} = \{x_{2}=f,y_{2}=rf,z_{2}=r^3 f\}$, there is no mapping $S_{1}\rightarrow S_{2}$ that will preserve all relations. This is due to the fact that in $S_{1}$, all of the products $(x_{1}y_{1}),(x_{1}z_{1})$, and $(y_{1}z_{1})$ have order $n=9$, while in $S_{2}$ the product $(x_{2}z_{2})$ has order $3<n=9$.
\end{examp}

These examples also serve to illustrate that the order in which generators are taken does matter. Furthermore, these examples illustrate the necessity of several of the relations we have imposed. However, even with all of this in mind, it will come to light in the next section that not all relations play a role in the actual computation of the $\lambda_{i}$ values.

\section{Bounds for $\lambda_{i}$ values}
\label{lambdaBounds}

In this section we discuss computing bounds for $\lambda_{1}$ and $\lambda_{2}$ corresponding to each of the presentations for $D_{n}$ given in Section \ref{presClass}. We apply the results from Section \ref{vonDyckIso} in order to work with concrete generating sets for $D_{n}$. We begin by computing bounds for the $\lambda_{i}$ in the case where the symmetric generating set has cardinality equal to two.

\subsection{Cardinality Two Case}
\label{lambdaBoundsA}

Fix an involution $f$ of $D_{n}$ and set $S_f = \{f,rf\}$, where $r$ generates the order $n$ cyclic subgroup of $D_{n}$. Clearly $S_f$ generates $D_{n}$ and $f^2=(rf)^2 = 1$, that is, $f$ and $rf$ are both involution elements of $D_{n}$. Thus the set $S_f$ is a symmetric generating set for $D_{n}$. We call this $S_f$ the symmetric order two simple generating set for $D_{n}$. Furthermore, it is the case that there is no smaller symmetric generating set for $D_{n}$. The following results hold:
\begin{theorem}
\label{twoFlipsSpecial}
 Let $G=D_{n}$ and $S_f = \{f,rf\}$, then
   \begin{description}
     \item[(a)] $\lambda_{1}(D_{n},S_f) \leq n$,
     \item[(b)] $\lambda_{2}(D_{n},S_f)=2$.
   \end{description}
\end{theorem}

\begin{proof} We begin by proving
Theorem \ref{twoFlipsSpecial} (a): In order to establish this result, we make the two following observations. First, since $f,rf$ are both involution elements, if $s=f$ or if $s=rf$ then $(g^{-1}sg)(g^{-1}sg)=1$, thus $g^{-1}sg$ is also an involution. Hence, either $g^{-1}sg=r^{\frac{n}{2}}$ or $g^{-1}sg=r^{i}f$ for some integer $i$, since these are the only involution elements in $D_n$. Second, again since $f,rf$ are both involution elements, words in the generators $f,rf$ must be alternating. So it is not difficult to construct Tables \ref{table1} and \ref{table2}, tables of words where the $i$th entry is the $i$th alternating product $s_1s_2s_1 \ldots$ which is $(s_1s_2)^{\frac{i}{2}}$ for $i$ even and is $(s_1s_2)^{\frac{i-1}{2}}s_1$ for $i$ odd.

\begin{table}[h]
\caption {Alternating products of generators, $n$ odd.}
    \centering
    \bgroup
\def\arraystretch{1.5}
\begin{tabular}{|c|c|c|}\hline
   1& $f$ & $rf$                                               \\ \hline
   2& $frf=r^{n-1}=r^{-1}$ & $rff=r$                         \\ \hline
   3& $frff=fr$ & $rffrf=r^2 f$                              \\ \hline
   4& $frffrf=r^{n-2}=r^{-2}$ & $rffrff=r^2$              \\ \hline
   5& $frffrff=fr^{2}$ & $rffrffrf=r^3 f$                  \\ \hline
   \vdots&\vdots & \vdots                                     \\ \hline
   $n$& $=fr^{\frac{n-1}{2}}$ & $=r^{\frac{n+1}{2}}f$         \\ \hline
\end{tabular}
\egroup
\label{table1}
\end{table}

\begin{table}[h]
\caption {Alternating products of generators,  $n$ even.}
    \centering
    \bgroup
\def\arraystretch{1.5}
\begin{tabular}{|c|c|c|}\hline
   1                                           &    $f$ & $rf$  \\ \hline
   2                       &    $frf=r^{n-1}=r^{-1}$ & $rff=r$  \\ \hline
   3                              &    $frff=fr$ & $rffrf=r^2 f$  \\ \hline
   4              &    $frffrf=r^{n-2}=r^{-2}$ & $rffrff=r^2$  \\ \hline
   5                    &    $frffrff=fr^{2}$ & $rffrffrf=r^3 f$  \\ \hline
   \vdots                                 &    \vdots & \vdots \\ \hline
   $n$        &    $=r^{-\frac{n}{2}}$ &  $=r^{\frac{n}{2}}$ \\ \hline
\end{tabular}
\egroup
\label{table2}
\end{table}

There is an important fact regarding the entries in Tables \ref{table1} and \ref{table2}: In the odd case $fr^{\frac{n-1}{2}}=fr^{\frac{n+1}{2}}$ and in the even case $r^{-\frac{n}{2}}=r^{\frac{n}{2}}$ but otherwise all of the other elements that appear are distinct. Thus, the entries in the table together with the identity element exhausts the list of elements for $D_{n}$. Now the longest element that appears has length $n$ which shows that $\lambda_{1}(D_{n},S)\leq n$ and thereby establishes part (a).

Proof for Theorem \ref{twoFlipsSpecial} (b): To compute $\lambda_2$, we need $l_s(g^{-1}ss'g)$.  However, if $s=s'$, $l_s=1$ since each element in $S$ is symmetric.\\
  Therefore, there are two remaining cases, each containing two sub-cases.\\
\begin{description}
\item{Case 1:} $ss'=r^{-1}$
\begin{description}
\item {$g=r^if$:} Then $g^{-1}ss'g=r$, therefore, $l_s=2$.
\item{$g=r^i$:} Then $g^{-1}ss'g=r^{-1}$, therefore, $l_s=2$.
\end{description}
\item{Case 2:} $ss'=r$
  \begin{description}
\item {$g=r^if$:} Then $g^{-1}ss'g=r^{-1}$, therefore, $l_s=2$.
\item{$g=r^i$:} Then $g^{-1}ss'g=r$, therefore, $l_s=2$.
\end{description}
\end{description}
Therefore $\lambda_2(D_n,S)=2$  for all $n$.
\end{proof}

{\bf Remark:} It is actually possible to say more than what is stated in Theorem \ref{twoFlipsSpecial} (a).  First note that equality
in (a) may be achieved if $n+1$ is divisible by 4. To see this, set $s=f$ and $g^{-1}=r^{\frac{n+1}{4}}$. Then
\begin{align*}
   g^{-1}sg&=r^{\frac{n+1}{4}}fr^{-\frac{n+1}{4}}, \\
   &=r^{2\frac{n+1}{4}}f, \\
   &=r^{\frac{n+1}{2}}f,
\end{align*}
and where $\frac{n+1}{2}$ is even. The proof of (a) and in particular Tables \ref{table1} and \ref{table2} shows that this gives an element of length $n$. This is an interesting contrast with the result stated in Theorem 1 (ii) in \cite{MoultonSteel2012}, where both the order of the group and the generating set increase with $n$; whereas if $G=D_n$ and $S_f = \{f,rf\}$, only the order of the group increases with $n$. Furthermore, note that if $s=f$ then $gsg^{-1}$ produces an element of $D_{n}$ of the form $r^{2k}f$, and if $s=rf$ then $gsg^{-1}$ produces an element of $D_{n}$ of the form $r^{2k+1}f$.  Thus, upon conjugating $s=f$ and $s=rf$ with each element of $D_{n}$ we obtain each of the involution elements of $D_{n}$ except $r^{\frac{n}{2}}$. Taking this into account, Tables \ref{table1} and \ref{table2} then show that $\lambda_{1}(D_{n},S_f) = n$ if $n$ is odd and $\lambda_{1}(D_{n},S_f) = n-1$ if $n$ is even.

Now, applying Lemma \ref{lengthLemma} to Theorem \ref{twoFlipsSpecial}, we obtain the following result:

\begin{theorem}
\label{twoFlipsGeneral}
  For any order two symmetric generating set $S = \{r^{a}f,f\}$ of $D_{n}$, we have
   \begin{description}
     \item[(a)] $\lambda_{1}(D_{n},S) \leq n$,
     \item[(b)] $\lambda_{2}(D_{n},S)\leq 2$,
   \end{description}
   and furthermore, $\lambda_{1}(D_{n},S)=n$ if $n$ is odd and $\lambda_{1}(D_{n},S)=n-1$ if $n$ is even.
\end{theorem}

 \subsection{Cardinality Three Cases}
  \label{lambdaBoundsB}

  In this section we consider the bounds for $\lambda_{1}(D_{n},S)$ and $\lambda_{2}(D_{n},S)$ for the cases when the generating set $S$ has three distinct elements. These cases correspond to the presentations described by Theorems \ref{DnPres2}, \ref{DnPres3}, and \ref{threeFlipsMain} (\ref{DnPres4})-(\ref{DnPres7}).

 \subsubsection{Two involutions and one cyclic element}
  \label{lambdaBoundsB1}

  The first case we consider is when the generating set $S$ is composed of three involutions, where exactly one of which is also contained in the order $n$ cyclic subgroup. This case is only relevant whenever $n$ is even. Furthermore, the result we obtain shows how the addition of a single generator may sometimes have a significant impact on the lengths of group elements. One may view this as an instance of what can happen with regard to the values for $\lambda_{1}(G,S)$ and $\lambda_{2}(G,S)$ when one makes a slight change in the presentation of the group $G$. We have the following result.

  \begin{theorem}
\label{twoFlipsOneRotSpecial}
 Let $G=D_{n}$ and $S_f = \{f,rf,r^{\frac{n}{2}}\}$, then
   \begin{description}
     \item[(a)] $\lambda_{1}(D_{n},S_f) = \frac{n}{2}$,
     \item[(b)] $\lambda_{2}(D_{n},S_f)=\lambda_{1}(D_{n},S_f)$.
   \end{description}
\end{theorem}

\begin{proof}
  To obtain the value for $\lambda_{1}(D_{n},S_f)$, we essentially look at the lengths of elements in Tables \ref{table1} and \ref{table2} and then examine how we may use $r^{\frac{n}{2}}$ to reduce the number of generators required to write each element of the group as product from $S_f = \{f,rf,r^{\frac{n}{2}}\}$. First, observe that for any $s\in S_{f}$ we have that $gsg^{-1}$ is either $r^{\frac{n}{2}}$, which happens exactly when $s=r^{\frac{n}{2}}$, or $gsg^{-1}=r^{k}f$ for some $0\leq k \leq n-1$. Furthermore, since $r^{k}f = fr^{-k} = fr^{n-k}$ and, with respect to $\{f,rf\}$, the length of $r^{k}f$ is the same as the length of $fr^{k-1}$, it suffices to consider only elements of the form $r^{k}f$ with $0\leq k \leq \frac{n}{2}$. Now we can begin to list these elements and their lengths with respect to $S_f = \{f,rf,r^{\frac{n}{2}}\}$. Notice that $f,rf$ have length one and $r^{\frac{n}{2}}f$ has length two. Furthermore, if we multiply by $r$ or $r^{-1}$ we add two to the length. Then, provided that $n>4$ is even, we have
  \begin{align*}
     &f  &\mbox{length 1}  \\
     &rf  &\mbox{length 1} \\
     &r^{\frac{n}{2}}f  &\mbox{length 2} \\
     &r^{2}f  &\mbox{length 3} \\
     &r^{\frac{n}{2}-1}f  &\mbox{length 4} \\
     &\vdots &\vdots \\
     &r^{\lfloor \frac{n}{4}\rfloor + 1}f  &\mbox{length $\frac{n}{2}$}
  \end{align*}
  which shows that the longest such element has length $\frac{n}{2}$. If $n=2,4$ one can see directly that $\lambda_{1}(D_{n},S_f) =2$.

  To obtain the value for $\lambda_{2}(D_{n},S_f)$, begin by noticing that $gss'g^{-1} = r,r^{-1}$ if $s,s'\in\{f,rf\}$ and $s\neq s'$, otherwise $gss'g^{-1}$ can be any element of $D_{n}$ of the form $r^{2k\pm \frac{n}{2}}f$ or $r^{2k+1\pm\frac{n}{2}}f$. Therefore, we can write any element of $D_{n}$ of the form $r^{p}f$ as $gss'g^{-1}$ where $s,s'\in \{f,rf,r^{\frac{n}{2}}\}$. This proves that $\lambda_{2}(D_{n},S_f)=\lambda_{1}(D_{n},S_f)$.
\end{proof}

Now, applying Lemma \ref{lengthLemma} to Theorem \ref{twoFlipsOneRotSpecial}, we obtain the following result.

\begin{theorem}
\label{twoFlipsOneRotGeneral}
 Let $G=D_{n}$ and $S_f = \{f,r^{a}f,r^{\frac{n}{2}}\}$, where $r^{a}$ generates an order $n$ cyclic subgroup, then
   \begin{description}
     \item[(a)] $\lambda_{1}(D_{n},S_f) = \frac{n}{2}$,
     \item[(b)] $\lambda_{2}(D_{n},S_f)=\lambda_{1}(D_{n},S_f)$.
   \end{description}
\end{theorem}

 \subsubsection{Two cyclic elements and one involution}
 \label{lambdaBoundsB2}

 Let $G=D_{n}$. Set $S_f = \{f,r,r^{n-1}\}$, where $r$ and $f$ are the group elements of $D_{n}$ previously described. Note $f^2=1$ and $rr^{n-1}=1$, thus the set $S_f$ is a symmetric generating set for $D_{n}$. We call this $S_f$ \emph{the simple chiral symmetric generating set for $D_{n}$}. The following results hold:
\begin{theorem}
\label{thmFour}
 Let $G=D_{n}$ and $S_f = \{f,r,r^{n-1}\}$, then
   \begin{description}
     \item[(a)] $\lambda_{1}(D_{n},S_f) = \left\{\begin{array}{ll}
     \lfloor \frac{n}{2} \rfloor + 1 & \mbox{if $4\divides n$,} \\
     \lfloor \frac{n}{2} \rfloor & \mbox{if $2\divides n$ and $4\ndivides n$,} \\
     \lfloor \frac{n}{2} \rfloor + 1 & \mbox{if $n$ is odd,}  \end{array}\right. $
     \item[(b)] $\lambda_{2}(D_{n},S_f) =
     \left\{\begin{array}{ll} 
      \lfloor \frac{n}{2} \rfloor & \mbox{if $4\divides n$,}\\
       \lfloor \frac{n}{2} \rfloor + 1 & \mbox{if $4\ndivides n$,}  \end{array}\right.$
   \end{description}
   where $\lfloor x \rfloor$ denotes the greatest integer $m \leq x$.
\end{theorem}

\begin{proof} We proceed by first looking for the elements of $D_{n}$ whose reduced word length, in terms of elements of $S_f$, is maximal. Choose a representative of such elements, call it $\phi_{0}$. We must determine if and how $\phi_{0}$ can be realized as a conjugate of the form $g^{-1}sg$ or $g^{-1}ss'g$.  If $\phi_{0}$ can be realized as a conjugate of the form $g^{-1}sg$ then we have found $\lambda_{1}$. On the other hand, if $\phi_{0}$ can be realized as a conjugate of the form $g^{-1}ss'g$ then we have found $\lambda_{2}$. If $\phi_{0}$ can't be realized in this conjugate form then we look for a next longest element of $D_{n}$, call it $\phi_{-1}$, and then proceed as just described. If necessary, continue to reduce to ``the next longest element'' until this process eventually ends and yields the values of $\lambda_{1}$ and $\lambda_{2}$.

Now, some elements with maximal length in $D_{n}$ with respect to the generating set $S _f$ are
\begin{align}
  \phi_{0}&= \left\{\begin{array}{ll} r^{\frac{n}{2}}f & \mbox{when $n$ is even,} \\ r^{\lfloor \frac{n}{2} \rfloor}f,\ \text{or}\ r^{\lceil \frac{n}{2} \rceil}f & \mbox{when $n$ is odd,}  \end{array} \right.
\end{align}
with $l_{S}\left( r^{\frac{n}{2}}f \right)=l_{S}\left( r^{\lfloor \frac{n}{2} \rfloor}f \right)=l_{S}\left( r^{\lceil \frac{n}{2} \rceil}f \right)=\lfloor \frac{n}{2} \rfloor + 1$. These facts come from a direct examination of the elements of $D_{n}$ as listed in (\ref{eq:DnDescription}). 

For $s=f$, and $g^{-1}=r^{k}$ with $0\leq k \leq \lfloor \frac{n}{2} \rfloor$, we have that
\begin{align}
 g^{-1}sg  &= g^{-1}fg  \\
           &= r^{k}fr^{-k}, \\
           &= r^{2k}f, \label{eq:2k}
\end{align}
where we have made use of the identities from (\ref{eq:rInverses})-(\ref{eq:flipSymmetry}).

\begin{case}
Observe that if $n$ is divisible by 4 then we can choose $k=\frac{n}{4}$ and achieve $g^{-1}sg=\phi_{0}=r^{\frac{n}{2}}f$. This shows that $\lambda_{1}(D_{n},S) = \frac{n}{2}+1$ whenever $4\divides n$.
\end{case}

\begin{case}
On the other hand, if $n$ is even and not divisible by 4 then it is not possible to achieve $\phi_{0}=r^{\frac{n}{2}}f$ with an element of the form $g^{-1}sg$. Instead, take $k=\lfloor \frac{n}{4} \rfloor$ which gives $g^{-1}sg=r^{\frac{n}{2}-1}f=\phi_{-1}$, the next longest element of $D_{n}$ so that in this case $\lambda_{1}(D_{n},S) = \frac{n}{2}$.
\end{case}

 \begin{case}
 If $n$ is odd, then take $k=\lfloor \frac{n}{4} \rfloor$ and we get $g^{-1}sg = r^{\lfloor \frac{n}{2} \rfloor}f = \phi_{0}$ so again $\lambda_{1}(D_{n},S) = \frac{n}{2}+1$.
 \end{case}

 This proves theorem \ref{thmFour} (a).

To obtain the result of theorem \ref{thmFour} part (b), the proof is similar to that for part (a) except that now the relevant conjugates of the form $g^{-1}ss'g$ that lead to either an element of maximal length $\phi_{0}$ (only possible if $n$ is odd) or a next longest element $\phi_{-1}$ of $D_{n}$, which are of the form $r^{2k+1}f$ or $r^{2k-1}f$. Notice these elements are products of $f$ with odd powers of $r$. These come from taking $g^{-1}=r^{-k}$ with $0\leq k \leq \lfloor \frac{n}{2} \rfloor$ and either $s=r$, $s'=f$; or $s=r^{-1}$, $s'=f$.

Finally, observe that if $n$ is divisible by 4 then the longest element is an even multiple of $r$ times $f$ and thus is not realizable as $g^{-1}ss'g$. However, in this case a next longest element is realizable as $g^{-1}ss'g$  so that $\lambda_{2}(D_{n},S) = \frac{n}{2}$ whenever $4\divides n$. If $n$ is even but not divisible by 4 then the longest element $\phi_{0}$ is a product of $f$ with an odd power of $r$ and thus is realizable as a conjugate of the form $g^{-1}ss'g$. Thus, in the case $n$ is even and $4 \ndiv n$ we have $\lambda_{2}(D_{n},S) = \frac{n}{2}+1$.  Lastly, if $n$ is odd then since $l_{S}\left( r^{\lfloor \frac{n}{2} \rfloor}f \right)=l_{S}\left( r^{\lceil \frac{n}{2} \rceil}f \right)=\lfloor \frac{n}{2} \rfloor + 1$ one of $\lfloor \frac{n}{2} \rfloor$ or $\lceil \frac{n}{2} \rceil$ is odd and hence is realizable as $g^{-1}ss'g$. Therefore $\lambda_{2}(D_{n},S) =\lfloor \frac{n}{2} \rfloor + 1$.
 \end{proof}

Combining the result of Theorem \ref{thmFour} together with the remark following Lemma \ref{lengthLemma} from section \ref{vonDyckIso} gives the following:
\begin{theorem}
\label{thmFive}
 Let $G=D_{n}$ and $S = \{r^{a}f,r^{b},r^{-b}\}$ and suppose that $r^{b}$ generates an order $n$ cyclic subgroup of $D_{n}$, then
   \begin{description}
     \item[(a)] $\lambda_{1}(D_{n},S_f) = \left\{\begin{array}{ll}
     \lfloor \frac{n}{2} \rfloor + 1 & \mbox{if $4\divides n$,} \\
     \lfloor \frac{n}{2} \rfloor & \mbox{if $2\divides n$ and $4\ndivides n$,} \\
     \lfloor \frac{n}{2} \rfloor + 1 & \mbox{if $n$ is odd,}  \end{array}\right. $
     \item[(b)] $\lambda_{2}(D_{n},S_f) =
     \left\{\begin{array}{ll} 
      \lfloor \frac{n}{2} \rfloor & \mbox{if $4\divides n$,}\\
       \lfloor \frac{n}{2} \rfloor + 1 & \mbox{if $4\ndivides n$,}  \end{array}\right.$
   \end{description}
   where $\lfloor x \rfloor$ denotes the greatest integer $m \leq x$.
\end{theorem}
This theorem establishes the values of $\lambda_{1}(D_{n},S)$ and $\lambda_{2}(D_{n},S)$ for any presentation with a form as in Theorem \ref{DnPres3}.

 \subsubsection{Three involution elements}
 \label{lambdaBoundsB3}

 In this section we discuss bounds for $\lambda_{1}(D_{n},S)$ and $\lambda_{2}(D_{n},S)$ with $S$ a generating set of the form described in Theorem \ref{threeFlipsMain} (\ref{DnPres4}) - (\ref{DnPres7}). While proving a bound for $\lambda_{2}(D_{n},S)$ in this case is straightforward, doing so for $\lambda_{1}(D_{n},S)$ is challenging and a complete proof is currently elusive. Therefore, in this section we present, by way of Theorem  \ref{3flip}, strong evidence for the following conjecture:

 \begin{conj}
For a generating set $S$ composed of three involutions, none of which belong to the order $n$ cyclic subgroup of $D_n$, we conjecture that $\lambda_{1}(D_{n},S) \leq \lfloor \frac{n}{2} \rfloor  + 1$.
\end{conj}

Before describing the evidence for this conjecture we discuss some properties of the sets $S_{1} = \{f,rf,r^{2}f\}$ and $S_{2}=\{f,rf,r^{3}f\}$.  Since for any $n\geq 2$ the pair $\{f,rf\}$ generates $D_{n}$ we have that $S_{1}$ generates $D_{n}$ if $n>2$, and $S_{2}$ generates $D_{n}$ if $n > 3$. Furthermore,  depending on whether $n$ is odd or even,  $S_{1}$ is representative of the case in Theorem \ref{threeFlipsMain} (\ref{DnPres4}) or (\ref{DnPres5}) respectively, and depending on $n$, $S_{2}$ could correspond to any of the cases in Theorem \ref{threeFlipsMain} (\ref{DnPres4}) - (\ref{DnPres6}). We will show that  $\lambda_{1}(D_{n}S_{1}),\lambda_{1}(D_{n},S_2)\leq \lfloor \frac{n}{2} \rfloor  + 1$ thereby establishing

 \begin{theorem}\label{3flip}
  Consider the dihedral group $D_{n}$ with $n \geq 2$. Then there exists a generating set of the form $S=\{f,r^{a}f,r^{b}f\}$ such that $\lambda_{1}(D_n,S) \leq \lfloor \frac{n}{2} \rfloor  + 1$, where $a\neq b$ and $a,b \geq1$.
\end{theorem}

Theorem \ref{3flip} is an immediate consequence of either of Lemma \ref{3fliplem1} or Lemma \ref{3fliplem2} below.  However, before establishing these two lemmas we make some observations regarding the two sets $S_{1}$ and $S_{2}$. There are three points that are useful to note:
\begin{enumerate}
  \item[(i)] The calculations for $\lambda_{1}(D_{n},S)$ where $S=S_{1}$ or $S=S_{2}$ are independent of $n$.
  \item[(ii)] In order to obtain a bound on the length of the elements of the order $n$ cyclic subgroup of $D_{n}$, it suffices to do so for each $r^{m}$ where $0 < m \leq \lfloor \frac{n}{2} \rfloor$. This is because an element and its inverse have the same length, so once we have the length for the first $\lfloor \frac{n}{2} \rfloor$ powers of $r$ we obtain the others using the inverse property.
  \item[(iii)] In fact, once you obtain a bound on the lengths of the powers of $r$, you get a bound on the lengths of flips by way of the following computation. Let $S = \{f,r^{a}f,r^{b}f\}$. Consider $r^{m}f$ then
      \[ r^{m}f = r^{m-a}\left(r^{a}f \right) = r^{m-b}\left(r^{b}f \right). \]
      Now, use the bound on $r^{m}$, $r^{m-a}$, or $r^{m-b}$, whichever gives the shortest length, then add one to account for $f$, $r^{a}f$, or $r^{b}f$, whichever is appropriate.
\end{enumerate}
Observation (iii) yields the following conjecture, whose truth does not influence the results of Lemma \ref{3fliplem1} or Lemma \ref{3fliplem2} below.
      \begin{conj}
      For all $0< m < n$, $l_S(r^{m}f) \leq  \max\limits_{0 \leq k <n}\left\{ l_S(r^k)\right\}$, where $l_S$ is the length defined in previous sections.
      \end{conj}
      There is strong evidence for this conjecture. Now we proceed to prove the aforementioned lemmas.  Based on the preceding observations, we only need to compute the lengths $r^{m}$ with $0 < m \leq \lfloor \frac{n}{2} \rfloor$. If the conjecture is correct then in some cases we can tighten the bound by one.

\begin{lemma}\label{3fliplem1}
	Let $S=\{f,rf,r^{2}f\}$.  Then for any $n> 2$, $S$ generates $D_n$ and $\lambda_{1}(D_n,S) \leq \lfloor \frac{n}{2} \rfloor  + 1$.
\end{lemma}
\begin{proof} For $S = \{f,rf,r^{2}f\}$,  let $A = \left(rf\right)\left(f\right) = r$ and $B = \left(r^{2}f\right)\left(f\right) = r^{2}$. Now, for an integer $0<m \leq \lfloor \frac{n}{2} \rfloor$ we split into two cases, $m$ even, and $m$ odd. If $m$ is even, then
      \[ r^{m} = B^{\frac{m}{2}}, \]
      which has length $l \leq 2 \frac{m}{2} \leq \lfloor \frac{n}{2} \rfloor$. On the other hand, if $m$ is odd then
      \[ r^{m} = r^{m-1}r = B^{\frac{m-1}{2}}A, \]
      where we have used the fact that $m$ odd implies $m-1$ is even. Now $B^{\frac{m-1}{2}}$ has length less or equal to $2 \frac{m-1}{2} = m-1$ and $A$ has length two. Thus we see that in case $m$ is odd, then $r^{m}$ has length $l \leq m+1 \leq \lfloor \frac{n}{2} \rfloor + 1$.

      Using the points noted above we can already see that for this generating set we have that $\lambda_{1} \leq \lfloor \frac{n}{2} \rfloor + 2$. In this case, we can tighten this bound by one.

      Next, consider $r^{p}f$. Again we split computations into the even and odd cases. If $m$ is even, then write
      \[ r^{p}f = r^{p-2}r^{2}f. \]
      Now since $p-2$ is even, by the previous result with $m=p-2$ we see that the length of $r^{p}f$ satisfies $l\leq p-2 + 1 = p-1 \leq \lfloor \frac{n}{2} \rfloor - 1$. If $p$ is odd then
      \[ r^{p}f = r^{p-2}r^{2}f, \]
      with $p-2$ odd. Thus, the length of $r^{p}f$ will satisfy $l \leq p-2 + 1 = p-1 \leq \lfloor \frac{n}{2} \rfloor + 1 - 1 = \lfloor \frac{n}{2} \rfloor$.

      This finally shows that for any value of $n > 2$ if we take the generating set $S = \{f,rf,r^{2}f\}$ for $D_{n}$ then $\lambda_{1} \leq \lfloor \frac{n}{2} \rfloor + 1$.
\end{proof}

\begin{lemma}\label{3fliplem2}
	Let $S=\{f,rf,r^{3}f\}$.  Then for any $n> 3$, $S$ generates $D_n$ and $\lambda_{1}(D_n,S) \leq \lfloor \frac{n}{2} \rfloor  + 1$.
\end{lemma}
\begin{proof} For $S = \{f,rf,r^{3}f\}$, let $A = \left(rf\right)\left(f\right) = r$, $B = \left(r^{3}f\right)\left(f\right) = r^{3}$, and $C = \left(r^{3}f\right)\left(rf\right) = r^{2}$.  For an integer $0<m \leq \lfloor \frac{n}{2} \rfloor$ we again split the calculations into two cases, $m$ even, and $m$ odd. If $m$ is even, then
      \[ r^{m} = C^{\frac{m}{2}}, \]
      which has length $l \leq 2 \frac{m}{2} \leq \lfloor \frac{n}{2} \rfloor$. On the other hand, if $m$ is odd then either $m=1$, in which case $r^{m} = r$ has length 2; $m=3$, in which case $r^{m} = r^{3}$ has length 2; or $m > 3$. If $m > 3$ is odd, then $m-3$ is even. Thus, since
      \[ r^{m} = r^{m-3}r^{3} = r^{m-3}B,\]
      and $m-3$ is even, we have that $r^{m}$ has length $l \leq m-3 + 2 \leq m - 1 \leq \lfloor \frac{n}{2} \rfloor - 1$.

      From this we already see that if $n > 3$, then with $S = \{f,rf,r^{3}f\}$ generating $D_{n}$ we have that $\lambda_{1} \leq \lfloor \frac{n}{2} \rfloor  + 1$.
\end{proof}

\noindent {\bf Remark: }There are examples to illustrate that, at least for some values of $n$, this the minimal upper bound. For a specific example take the following.
\begin{examp}
	Consider $n=3$ with $S = \{f,rf,r^{2}f\}$. Then
 	$D_{3} = \{1,r=(rf)(f),r^{2}=(r^{2}f)(f),f,rf,r^{2}f\}$ and thus $\lambda_{1} = 2 = \lfloor \frac{n}{2} \rfloor  + 1$.
\end{examp}

We currently have not discovered a way to construct, for \emph{any} $n$, a set of the form $S = \{f, r^af, r^bf\}$ such that $\lambda_1(D_n,S) \leq \lfloor \frac{n}{2} \rfloor  + 1$.

To conclude this section, we now obtain a bound on $\lambda_{2}(D_{n},S)$.

\begin{theorem}
  Let $S = \{f,r^{a}f,r^{b}f\}$, then $\lambda_{2}(D_{n},S) = 2$.
\end{theorem}
\begin{proof}
Direct computation shows that $gss'g^{-1}$ with $s,s'\in S$ produces each of $1,r,r^{-1},r^{a},r^{-a},r^{b},r^{-b},r^{a-b},r^{b-a}$ and the maximal length of these is two.
\end{proof}

\section{Conclusion}
\label{conclusion}

By classifying all possible finite presentations of the dihedral groups with symmetric generating sets of cardinality less than or equal to three, and establishing that the values of the quantities $\lambda_{1}(G,S)$ and $\lambda_{2}(G,S)$ defined in \cite{MoultonSteel2012} are preserved by certain automorphism that preserves relations, we have proven a family of bounds for $\lambda_{1}(G,S)$ and $\lambda_{2}(G,S)$ with $G=D_{n}$ and $S$ one of several different generating sets. These results serve to illustrate some of the characteristics of $\lambda_{1}(G,S)$ and $\lambda_{2}(G,S)$ that where merely hinted at in \cite{MoultonSteel2012}. One novel feature of our work is the utilization of the group presentation point of view. Thus, our approach may be adapted to other finitely presented groups. This is of interest in both theory and in applications. For example, since finitely presented groups play an important role in both combinatorial genomics and formal language theory, see \emph{e.g.}\ \cite{chiswell}, one may expect that the study of the quantities $\lambda_{1}(G,S)$ and $\lambda_{2}(G,S)$ for finitely presented groups should be relevant to those fields.

Of course there is more that one can say about $\lambda_{1}(G,S)$ and $\lambda_{2}(G,S)$ for $G=D_{n}$. First off, what if $S$ has cardinality greater than three? Besides the trivial case with $S=D_{n}-\{1\}$, already discussed in general in the introduction, one may consider an additional ``extreme'' case with $S = \{f,rf,\ldots,r^{n-1}f\}$. It is easy to see for this choice of $S$ that $\lambda_{1}(D_{n},S)=1$ and $\lambda_{2}(D_{n},S)=2$. Then, probably the most interesting remaining cases are when $3<|S|<n$. It is likely that each of these cases can be tackled using the approaches we have developed here. However, the computations quickly become prohibitively tedious. Thus, it is desirable, if possible, to have a more unified approach to computing $\lambda_{1}(G,S)$ and $\lambda_{2}(G,S)$, at least when $G=D_{n}$.

\section{Appendix}
\label{NumThrConditions}

Consider a dihedral group $D_{n}$ described as in equation (\ref{eq:DnDescription}). The goal of this appendix is to establish a number-theoretic condition on the exponents $a,b$ in $S = \{r^{a}f,r^{b}f\}$ to distinguish how many pairs it takes to generate. This allows one to realize the presentations described in section \ref{classSectionA} in a more concrete manner. More significantly, this will allow for the direct computation of bounds for the quantities $\lambda_{1}(D_{n})$ and $\lambda_{2}(D_{n})$.

We begin with the observation that, given $S = \{r^{a}f,r^{b}f\}$, we can change notation by defining $r^{\tilde{a}} = r^{b-a}$ and $\tilde{f} = r^{a}f$, then $r^{b}f = r^{b-a}r^{a}f = r^{\tilde{a}}\tilde{f}$ and thus $S = \{\tilde{f},r^{\tilde{a}}\tilde{f}\}$. In light of this observation, to generate $D_{n}$ it suffices to establish conditions on an integer $a$ for a subset $S = \{f,r^{a}f\}$ of $D_{n}$, with $f^{2}=1$ and $r^{a}$ an element of the order $n$ cyclic subgroup of $D_{n}$. This will be the case if and only if the product $(r^{a}f)(f) = r^{a}$ generates an order $n$ cyclic subgroup, which will happen if and only if there is a solution to the equation $xa \equiv 1 \mod n$. From this we obtain the following result.

\begin{lemma}
\label{numCharOne}
A subset $S = \{f,r^{a}f\}$ of $D_{n}$, with $f^{2}=1$ and $r^{a}$ an element of the order $n$ cyclic subgroup of $D_{n}$, will generate $D_{n}$ if and only if $a$ is relatively prime with $n$.
\end{lemma}

Now it is easy to apply this lemma to obtain conditions on a subset $S=\{f,r^{a}f,r^{b}f\}$ of $D_{n}$, which correspond to the situations described abstractly in Theorem \ref{threeFlipsMain} (\ref{DnPres4})--(\ref{DnPres6}), to determine when one, two, or all three pairs from $S$ generate $D_{n}$. It only remains to give a concrete realization of the situation described abstractly by Theorem \ref{threeFlipsMain} (\ref{DnPres7}).

\begin{theorem}
\label{numCharTwo}
For a subset $S=\{f,r^{a}f,r^{b}f\}$ of $D_{n}$, with $f^{2}=1$ and $r^{a},r^{b}$ belonging to the order $n$ cyclic subgroup of $D_{n}$, it is the case that none of the pairs $\{f,r^{a}f\},\{f,r^{b}f\},\{r^{a}f,r^{b}f\}$ generate $D_{n}$ and the triple $\{f,r^{a}f,r^{b}f\}$ does, if and only if the following hold:
\begin{enumerate}
 \item $a$ is not relatively prime with $n$
 \item $b$ is not relatively prime with $n$
 \item $b-a$ is not relatively prime with $n$
 \item $a,b$ are relatively prime with one another.
\end{enumerate}
\end{theorem}

{\bf Proof:} Suppose that none of the pairs $\{f,r^{a}f\},\{f,r^{b}f\},\{r^{a}f,r^{b}f\}$ generate $D_{n}$, but the triple $\{f,r^{a}f,r^{b}f\}$ does. It is then clear from Lemma \ref{numCharOne} that none of $a$, $b$, and $b-a$ are relatively prime with $n$. Now, $\langle f,r^af,r^bf \rangle \subseteq \langle r^{as_1+bt_1},r^{as_2+bt_2}f \rangle$. Thus, $r \in \langle r^{as_1+bt_1},r^{as_2+bt_2}f \rangle$ implies that there exist $s_1,t_1$ such that $as_1+bt_1=1$, i.e. $a$ and $b$ are relatively prime.

On the other hand, if none of $a$, $b$, and $b-a$ are relatively prime with $n$, then it is clear from Lemma \ref{numCharOne} that none of the pairs $\{f,r^{a}f\},\{f,r^{b}f\},\{r^{a}f,r^{b}f\}$ can generate $D_{n}$. Suppose that $a,b$ are relatively prime. Then there exist integers $k,l$ such that $ka + lb = 1$ thus $(r^{a}f f)^{k}(r^{b}f f)^{l} = r^{ka + lb} = r$, from which it follows that $D_{n}$ can be generated.

\qed

\section*{Acknowledgements}
The authors extend their sincere gratitude to the editor and anonymous referee for valuable suggestions that have improved this manuscript. 

\bibliographystyle{elsarticle-num}
\bibliography{bioGroups}

\begin{thebibliography}{10}
\expandafter\ifx\csname url\endcsname\relax
  \def\url#1{\texttt{#1}}\fi
\expandafter\ifx\csname urlprefix\endcsname\relax\def\urlprefix{URL }\fi
\expandafter\ifx\csname href\endcsname\relax
  \def\href#1#2{#2} \def\path#1{#1}\fi

\bibitem{MoultonSteel2012}
V.~Moulton, M.~Steel, The `butterfly effect' in {C}ayley graphs with
  applications to genomics, J. Math. Biol. 65~(6-7) (2012) 1267--1284.

\bibitem{BafnaPevzner1996}
V.~Bafna, P.~A. Pevzner, Genome rearrangements and sorting by reversals, SIAM
  J. Comput. 25~(2) (1996) 272--289.

\bibitem{Bergeron2005}
A.~Bergeron, A very elementary presentation of the {H}annenhalli-{P}evzner
  theory, Discrete Appl. Math. 146~(2) (2005) 134--145.

\bibitem{EgriNagy2014}
A.~Egri-Nagy, V.~Gebhardt, M.~M. Tanaka, A.~R. Francis, Group-theoretic models
  of the inversion process in bacterial genomes, J. Math. Biol. 69~(1) (2014)
  243--265.

\bibitem{Hannenhalli1999}
S.~Hannenhalli, P.~A. Pevzner, Transforming cabbage into turnip: polynomial
  algorithm for sorting signed permutations by reversals, J. ACM 46~(1) (1999)
  1--27.

\bibitem{coxeter1980}
H.~S.~M. Coxeter, W.~O.~J. Moser, Generators and relations for discrete groups,
  4th Edition, Vol.~14 of Ergebnisse der Mathematik und ihrer Grenzgebiete
  [Results in Mathematics and Related Areas], Springer-Verlag, Berlin-New York,
  1980.

\bibitem{johnson1997}
D.~L. Johnson, Presentations of groups, 2nd Edition, Vol.~15 of London
  Mathematical Society Student Texts, Cambridge University Press, Cambridge,
  1997.

\bibitem{Smith2000}
G.~Smith, O.~Tabachnikova, Topics in group theory, Springer Undergraduate
  Mathematics Series, Springer-Verlag London, Ltd., London, 2000.

\bibitem{DummitFoote2004}
D.~S. Dummit, R.~M. Foote, Abstract algebra, 3rd Edition, John Wiley \& Sons,
  Inc., Hoboken, NJ, 2004.

\bibitem{carter2009}
N.~C. Carter, Visual group theory, Classroom Resource Materials Series,
  Mathematical Association of America, Washington, DC, 2009.

\bibitem{Humphreys1990}
J.~E. Humphreys, Reflection groups and {C}oxeter groups, Vol.~29 of Cambridge
  Studies in Advanced Mathematics, Cambridge University Press, Cambridge, 1990.

\bibitem{cunningham2014}
K.~K.~A. Cunningham, T.~Edgar, A.~G. Helminck, B.~F. Jones, H.~Oh, R.~Schwell,
  J.~F. Vasquez, On the structure of involutions and symmetric spaces of
  dihedral groups, Note Mat. 34~(2) (2014) 23--40.

\bibitem{fenrick1998}
M.~H. Fenrick, Introduction to the {G}alois correspondence, 2nd Edition,
  Birkh\"auser Boston, Inc., Boston, MA, 1998.

\bibitem{chiswell}
I.~Chiswell, A course in formal languages, automata and groups, Universitext,
  Springer-Verlag London, Ltd., London, 2009.

\end{thebibliography}

\end{document}